\documentclass[11pt]{article}
\usepackage{newcommands2006, amsfonts, amsmath, epic, amssymb, amsthm, 
array, rotating, bm}

\newtheorem{theorem}{Theorem}[section] 
\newtheorem{prop}[theorem]{Proposition}
\newtheorem{lemma}[theorem]{Lemma}

\theoremstyle{definition} 
\newtheorem{example}[theorem]{Example}

\theoremstyle{remark} 
\newtheorem{remark}[theorem]{Remark}

\newcommand{\bfRic}{\ensuremath{\mathbf{Ric}} }

\begin{document}
%\baselineskip = 1.3\baselineskip

\title{The Ricci Flow for Nilmanifolds}
\author{Tracy L. Payne }

%\address{Department of Mathematics, Idaho State 
%University, Pocatello, ID 83209-8085}
%\email{payntrac@isu.edu}
%\keywords{nilmanifold -- Ricci flow -- soliton metric}
%\subjclass[2000]{Primary: 53C44;  Secondary: 22E25.}

%\date{\today}

\maketitle  %% put this here for amsart class

\begin{abstract}
We consider the Ricci flow for simply connected nilmanifolds, which 
translates to a Ricci flow on the space of nilpotent metric Lie algebras.
We consider the evolution of the inner product with respect to
time and the evolution of  structure constants with respect to time,
as well as the evolution of these quantities modulo rescaling.  
We set up systems of O.D.E.'s for some of these flows and describe
their qualitative properties.
We also present some explicit solutions for the evolution of soliton
metrics under the Ricci flow.  
\end{abstract}

%\tableofcontents

\section{Introduction}\label{intro} 

The Ricci flow, defined by R. Hamilton (\cite{hamilton82}),  is an
important tool for  understanding the topology and geometry of
three-manifolds.  It is key in Perelman's  work
(\cite{perelman1}, \cite{perelman2},  \cite{perelman3}), and has
useful applications in other areas of  geometry as well.  Due to the
difficulty of solving the partial differential equations involved,
very few explicit examples of Ricci flow solutions are known.   

For a homogeneous manifold, the Ricci flow can be presented as a set
of O.D.E.'s rather than P.D.E.'s.  In many cases, it is  possible to
solve these systems exactly or to make estimates that allow a
description of the qualitative behavior of the system.   The Ricci
flows for the universal covers of compact homogeneous spaces of
dimension three have been  analyzed in \cite{isenberg-jackson-92} and
\cite{knopf-mcleod-01}.  Solutions  for certain metrics on the
universal covers of compact homogeneous four-manifolds are described
in \cite{isenberg-jackson-lu}.  

These examples provide insight into the behavior of the Ricci flow in
general, exhibiting many of the phenomena  of interest in  variable
curvature cases.  If a manifold $M$  admits a homogeneous metric, the
set of  homogeneous metrics on $M$ is invariant under    the Ricci
flow on the space of all metrics on $M$.  One hopes that this
invariant set is an attractor for the Ricci flow.   If that is true, a
description of the Ricci flow for homogeneous  metrics on $M$ would
yield an understanding of the long-term behavior of metrics in  a
large subset of the set of all metrics on $M$.

In this work, we continue the study of the Ricci flow for  homogeneous
spaces initiated in \cite{isenberg-jackson-92},
\cite{knopf-mcleod-01}, and \cite{isenberg-jackson-lu}.  We  analyze
the Ricci flow for the class of homogeneous spaces consisting of
simply connected nilmanifolds of arbitrary dimension.   A {\em
nilmanifold} is a Riemannian manifold with universal cover  $(N,g),$
where $N$ is a simply-connected nilpotent Lie group $N$ endowed with a
left-invariant metric $g$.  Although no Einstein metrics exist on a
nilpotent Lie group $N$ unless $N$ is abelian (Corollary 2,
\cite{jensen-69}), many  nilpotent groups admit  soliton metrics.  A
{\em soliton metric} is a  metric  $g$ that evolves under the Ricci
flow by diffeomorphisms and rescaling; that is,  $g_t = t \cdot
\eta_t^\ast g_0,$ where $\eta_t$ is a one-parameter family of
diffeomorphisms.   If a nilpotent Lie group does admit  a soliton
metric, then it is unique up to scaling (\cite{lauret01a}).  However,
not all nilpotent Lie groups admit soliton metrics.  Using   geometric
invariant theory, it can be shown that under the Ricci flow, any
left-invariant metric $g$ on a nilpotent Lie group $N$ approaches, modulo
rescaling,  a unique soliton metric $g^\prime$ on a
 nilpotent Lie group $N^\prime,$ and if $N$ admits a soliton metric, 
 that limiting nilmanifold is the soliton metric on $N.$
 (\cite{jablonski-08}).

The Ricci flow for three- and four-dimensional nilmanifolds is  fairly
well understood.   In dimension three, there is a single simply
connected  nonabelian nilpotent Lie group, the Heisenberg group $H_3.$
It is shown in   \cite{isenberg-jackson-92}  that under the Ricci
flow, any initial left-invariant metric $g_0$ on  $H_3$ collapses to a
flat metric on $\boldR^2$.  In dimension four, there is a single
simply connected nilpotent Lie group that is not a product of
lower-dimensional nilpotent Lie groups, the filiform group $L_4.$   It
was shown in  \cite{isenberg-jackson-lu}  that under the Ricci flow,
any initial metric $g_0$ that is diagonal with respect to a special
basis  collapses.   As time goes to infinity, any initial metric on
$H_3$ and the special metrics on $L_4$  can be viewed in the
appropriate framework as asymptotically projectively  approaching  the
unique soliton metric on $H_3$ or  $L_4$, respectively  (See
\cite{lott-05}, \cite{guenther-isenberg-knopf}).

Now we summarize the main results in this paper.   Following  this
section,  in Section \ref{preliminaries} we establish the necessary
background and  preliminaries involving nilmanifold geometry  and the
Ricci flow. The Ricci flow on the space of  left-invariant metrics on
a simply connected Lie group $G$ can be converted  to a  flow on the
space of inner products on  the corresponding  Lie algebra $\frakg.$
The flows for all individual metric Lie algebras can be  combined to
define a Ricci flow on the space of metric Lie algebras, and this flow
projects to a projectivized Ricci flow on the space of all volume-normalized
metric Lie algebras.    We define what it means for
for a metric Lie algebra to projectively approach another metric Lie
algebra, and what it  means  for a metric Lie algebra to collapse
under the Ricci flow.   We define a  Lie bracket flow for a Lie
algebra  that describes how the Lie bracket relative to a moving
orthonormal basis changes under the Ricci flow, and we define a
projectivization of this flow.

 In Theorem \ref{ricci-system} of  Section \ref{systems}, we  set up
systems of O.D.E.'s for the Ricci flow and the  Lie bracket flow   for
a single nilpotent Lie algebra $(\frakn,\sfQ).$  Theorem
\ref{projectivized} gives  O.D.E.'s for the projectivized Lie bracket
flow.  Proposition \ref{conserved} describes invariant quantities for
the Ricci flow for nilpotent metric Lie algebras.

In Section \ref{soltraj}, we find some explicit solutions for soliton
trajectories for the Ricci flow for nilpotent metric Lie algebras;
these are presented in Theorem \ref{soliton traj}.  The theorem only
gives solutions for nilpotent metric Lie algebras admitting a special
kind of basis.  We describe some broad conditions under which these
bases exist. 

Finally, in Section \ref{examples}, we give some examples.  We
consider the Lie bracket flow for Heisenberg Lie algebras and a Lie
algebra that does not  admit a soliton metric.  

The author is grateful to Jim Isenberg for his suggestion of studying
the Ricci flow for nilmanifolds. She thanks him, Jorge Lauret, Peng
Lu  and Dave Glickenstein for helpful discussions.
 This work was supported by NSF ADVANCE grant \#SBE-0620073.

\section{Preliminaries}\label{preliminaries}
\subsection{Structure of metric nilpotent Lie
algebras}\label{structure}

 Suppose that $(\frakg,\sfQ)$ is an $n$-dimensional metric Lie algebra
with basis $\calB = \{\bfx_i\}_{i=1}^n.$ The set
  \[ \Lambda_\calB = \{ (j,k,l) \, | \, \alpha_{jk}^l \ne 0 ,  1 \le j
< k  \le n, 1 \le l \le n  \}\] indexes the set of nonzero structure
constants  $\alpha_{jk}^l$ for $\frakg$ relative to $\calB$ without
repetitions due to skew-symmetry.  

Suppose that $\frakg$ is nonabelian, so $\Lambda_\calB$ is nonempty.
Let  $\{ \bfe_i \}_{i = 1}^n$ be the standard orthonormal basis for
$\boldR^n$.  For $1 \le j,k,l \le n$,  define the $1 \times n$ row
vector $\bm{y}_{jk}^l$ to be $\bfe_j^T + \bfe_k^T - \bfe_l^T.$  We
call a  vector $\bm{y}_{jk}^l,$ where $(j,k,l) \in \Lambda_\calB,$ a
{\em root vector} for $(\frakg,\sfQ)$ relative to the basis  $\calB$.
Let $\bm{y}_1,\bm{y}_2, \ldots, \bm{y}_m$ (where $m = | \Lambda_\calB
|$) be  an enumeration of the root vectors $\bfy_{jk}^l$ for
$(\frakg,\sfQ)$ relative to  $\calB$,  in the dictionary order on the
integer triples $(j,k,l)$.   Define the {\em root matrix} $Y$ for
$(\frakg,\sfQ)$ relative to  $\calB$ to be the $m \times n$ matrix
whose rows are the root vectors $\bm{y}_1, \bm{y}_2, \cdots,
\bm{y}_m.$  When $\frakg$ is nonabelian,  the {\em Gram matrix} $U$
for $(\frakg,\sfQ)$ relative to $\calB$  is defined to be the $m
\times m$ matrix $U = Y Y^T.$ 

Suppose that the basis $\calB$ is orthonormal, so that the inner
product $\sfQ$ can be written as  $\sfQ = \sum_{i=1}^n  \, dx^i
\otimes dx^i,$ where  $dx^i$ is dual to $\bfx_i,$ for $i=1, \ldots,
n.$   The structure constants for $\frakg$ relative to  $\calB$ are
given by
\begin{equation}\label{alpha-def} \alpha_{jk}^l =
\sfQ([\bfx_j,\bfx_k],\bfx_l) \end{equation} for  $1 \le j,k,l \le n.$
Letting  $\sfQ = \sum_{i=1}^n  q_i \, dx^i \otimes dx^i,$ where $q_1,
\ldots, q_n > 0,$ yields a family of inner products on $\frakg.$  For
each such $\sfQ,$ there is an orthonormal basis  $\overline{\calB_{\sfQ}} =
\overline{\calB} $
obtained by rescaling  each basis vector  in $\calB$ by its length;
elements of $\overline{\calB}$ are the vectors $\overline{\bfx_i} =
\frac{1}{\| \bfx_i\|} \, \bfx_i,$ for $i = 1, \ldots, n.$ The
structure constants for $\frakg$ relative to the orthonormal basis
$\overline{\calB}$ are  
\begin{equation}\label{a-tilde} \sfQ( [\overline \bfx_j,\overline
\bfx_k], \overline \bfx_l) = \frac{ \| \bfx_l\| }{ \| \bfx_j \| \cdot
\| \bfx_k \| } \, \alpha_{jk}^l = \sqrt{\frac{q_l}{q_j q_k}} \,
\alpha_{jk}^l .\end{equation} Note that the set
$\Lambda_{\overline{\calB}}$ defined by $\overline{\calB}$  is the
same as the set $\Lambda_{\calB}$ defined by $\calB,$ and hence the
set of  root vectors is the same for $\calB$ and $\overline{\calB}.$
Consequently $\calB$ and $\overline{\calB}$ have the same root matrix
$Y$ and the same Gram matrix $U.$

We define the {\em structure vector} for  $(\frakg,\sfQ)$ relative to
the orthogonal basis $\calB$ to be   the $m \times 1$ vector $\bfa$
having as entries the squares of the nonzero  structure constants
relative to  the orthonormal basis $\overline \calB.$ More precisely,
the   entries of $\bfa$ are the  numbers $\frac{q_l}{q_j q_k}
(\alpha_{jk}^l)^2 ,$ for $(j,k,l)$ in $\Lambda_\calB,$  listed in the
dictionary order on $(j,k,l)$.

\subsection{Curvatures of nilpotent metric Lie algebras} 

  We say that a  basis $\calB = \{ \bm{x}_i \}_{i=1}^n$ of  a metric
Lie algebra $(\frakg,\mathsf{Q})$ is a \textit{Ricci-diagonal} if  the
Ricci form for $(\frakg,\mathsf{Q})$ is diagonal when represented with
respect to $\calB.$  The $1 \times n$ vector $\bfRic_\calB$ defined by
\begin{equation}\label{def-of-Ric} \bfRic_\calB =
\left(\frac{\ric(\bfx_1,\bfx_1)}{\| \bfx_1 \|^2},
\frac{\ric(\bfx_2,\bfx_2)}{\| \bfx_2\|^2}, \ldots,
\frac{\ric(\bfx_n,\bfx_n)}{\| \bfx_n \|^2}\right)
 \end{equation} will be called the {\em Ricci vector} for
$(\frakn,\sfQ)$ relative to the basis $\calB.$    Together, an
orthogonal Ricci-diagonal basis $\calB$ for $(\frakn,\sfQ),$  its
Ricci vector, and the  set of lengths of the basis vectors determine
the inner product and Ricci form for  $(\frakn,\sfQ).$

The next theorem  gives  easy-to-compute   formulas for the Ricci form
and Ricci vector   for a metric nilpotent Lie algebra $(\frakn,\sfQ).$
Recall that the Lie bracket of a metric Lie algebra $(\frakg,\sfQ)$ is
encoded in the linear map $J : \frakg \to \End(\frakg)$ defined  by
$J_{\bfx}(\bfy) = \ad^\ast_{\bfy} \bfx ,$ for $\bfx$ and $\bfy$ in
$\frakg.$  The inner product $\sfQ(\cdot,\cdot)  = \la \cdot , \cdot
\ra$ on $\frakg$ induces an inner product on the tensor algebra of
$\frakg,$ which we also denote by $\la \cdot , \cdot \ra.$ 
 
\begin{theorem}[Theorems 6 and 8, \cite{payne05}]\label{riccitensor}
Let  $(\frakn,\sfQ)$  be a nonabelian metric nilpotent Lie algebra.
Then the Ricci form for $(\frakn,\sfQ)$ is given by
\begin{equation}\label{ricxy} \ric (\bfx, \bfy)  = - \smallfrac{1}{2}
\la \ad_{\bfx},\ad_{\bfy} \ra +  \smallfrac{1}{4} \la J_{\bfx},
J_{\bfy} \ra, 
\end{equation} where $\bfx$ and $\bfy$ are in $\frakn.$  Let $\calB$
be an orthogonal Ricci-diagonal basis with associated root matrix $Y$
and structure vector $\bfa.$   The  Ricci vector for $(\frakn,\sfQ)$
relative to $\calB$  may be written as 
\begin{align}\label{Ricci vector} \bfRic_{\calB} &= -\smallfrac{1}{2}
\sum_{(j,k,l) \in \Lambda_\calB} \frac{q_l}{q_jq_k}(\alpha_{jk}^l)^2
(\bm{y}_{jk}^l) \\ & = -\smallfrac{1}{2} \bfa^T Y. \notag
\end{align}
\end{theorem}

\begin{remark}\label{good} From Equation \eqref{ricxy} it follows that
an orthogonal  basis $\calB = \{ X_i\}$ is Ricci-diagonal if the sets
$\{ J_{X_i}\}$ and $\{ \ad_{X_i}\}$ are orthogonal.  
\end{remark} 

 Soliton inner products on nilpotent Lie algebras may be characterized
algebraically by the property that the Ricci endomorphism differs from
a scalar multiple of the identity automorphism  by a derivation  $D =
\Ric_\sfQ - \beta \Id$ of the Lie algebra (\cite{lauret01a}).  The
constant $\beta$ is called the {\em soliton constant} for
$(\frakn,\sfQ).$ The soliton constant is always negative when $\frakn$
is nonabelian, and the eigenvalues  for $D$ are positive and rational
(See \cite{heberinv}, \cite{lauret01a}). 

The next theorem gives an easily checked  linear condition that is
equivalent to  an inner product $\sfQ$ on a nilpotent Lie algebra
$\frakn$ being a soliton inner product.
\begin{theorem}[Theorem 1, \cite{payne05}]\label{Ua} Let
$(\frakn,\sfQ)$ be a  nonabelian nilpotent  metric Lie algebra with
orthogonal Ricci-diagonal basis $\calB.$ Let $U$ and $\bfa$ be the
Gram matrix and the structure vector for $(\frakn,\sfQ)$ with respect
to $\calB.$ Let  $m = | \Lambda_\calB |.$  Then $\sfQ$ is a soliton
inner product on $\frakn$ with nilsoliton constant $\beta$ if and only
if $U \bfa = -2\beta \onevector{m}.$  
\end{theorem} 

We conclude the section with some examples that illustrate the
definitions and theorems presented thus far. First we will consider
the indecomposable nilpotent Lie algebras in dimensions three and four.  
\begin{example}\label{h3-a}  Let  $\frakh_3$  be the three-dimensional
Heisenberg algebra, and let $\sfQ$ be an inner product on $\frakh_3.$
There exists an orthogonal basis $\calB = \{\bm{x}_i\}_{i=1}^3$ such
that  $[\bm{x}_1,\bm{x}_2]= \bm{x}_3.$ Then $\sfQ$ is of the form
$\sum_{i = 1}^3  q_i \,dx^i \otimes dx^i .$ There is a  single root
vector $\bm{y}_1 = \bfy_{12}^3 = (1,1,-1)$, and the Gram matrix with
respect to $\calB$ is $U  = [3]$.

Set $(q_1,q_2,q_3)=(1,1,1).$  Then the structure vector is $\bfa =
[1],$ and $U [\bfa] = 3 [1],$ so by Theorem \ref{Ua},
$(\frakh_3,\sfQ)$ is soliton with soliton constant  $\beta = -3/2$. By
Theorem \ref{riccitensor}, the Ricci vector for $(\frakh_3,\sfQ)$
relative to $\calB$ is $ -\smallfrac{1}{2}(1,1,-1).$ Therefore, with
respect to $\calB,$ the Ricci form is represented by a  diagonal
matrix with diagonal entries $-\smallfrac{1}{2}, -\smallfrac{1}{2},
\smallfrac{1}{2}.$ The  derivation  $D = \Ric + \frac{3}{2} \Id$
corresponding to the soliton inner product has eigenvectors $\bm{x}_1,
\bm{x}_2,$ and $\bm{x}_3$ with eigenvalues  $1, 1,$ and $2$
respectively.  The reader may check that any choice of positive $q_1,
q_2$ and $q_3$ yields a soliton inner product; all such metrics are
homothetic to $\sfQ = \sum_{i=1}^3  \,dx^i \otimes dx^i.$  
\end{example}

Next we consider the family of  four-dimensional filiform metric
nilpotent Lie algebras $(\frakl_4,\sfQ)$  studied in
\cite{isenberg-jackson-lu} (under the name A6).  
\begin{example}\label{l4-a} The nilpotent Lie algebra $\frakl_4$ can
be represented with  respect to the basis $\calB =
\{\bm{x}_i\}_{i=1}^4$ so that  all the Lie algebra relations are
determined by the relations $[\bm{x}_1,\bm{x}_2]= \bm{x}_3$ and
$[\bm{x}_1,\bm{x}_3]=\bm{x}_4.$ Define a family of inner products
$\sfQ (\cdot, \cdot ) = \sum_{i = 1}^4  q_i \, dx^i \otimes dx^i $ on
$\frakl_4.$  For all such $\sfQ,$ there are two root vectors $\bm{y}_1
= \bfy_{12}^3 = (1,1,-1,0)$ and  $\bm{y}_2 = \bfy_{13}^4 =
(1,0,1,-1),$ and  the Gram matrix is $U = (\begin{smallmatrix} 3 & 0
\\ 0 & 3  \end{smallmatrix})$.

If we define the inner product $\sfQ^\star$ by setting
$(q_1,q_2,q_3,q_4)=(1,1,1,1)$, then the structure vector for
$(\frakl_4,\sfQ^\star)$ relative to $\calB$ is $\bfa = \onevector{2},$
and $U \bfa = 3 \onevector{2},$ so by Theorem \ref{Ua},
$(\frakn,\sfQ)$  is soliton with soliton constant $\beta = -3/2$.
The Ricci vector is 
\[ \bfRic = -\smallfrac{1}{2} (\bfy_{12}^3 + \bfy_{13}^4) =
\smallfrac{1}{2}(-2,-1,0,1)\] and the corresponding derivation   $D =
\Ric + \frac{3}{2} \Id$ has eigenvalues
$\smallfrac{1}{2},\smallfrac{2}{2},\smallfrac{3}{2},$ and
$\smallfrac{4}{2}.$
\end{example}

 In the next example we  consider the five-dimensional Heisenberg Lie
algebra.
\begin{example}\label{h5-a} Let $(\frakh_5,\sfQ)$ be the
five-dimensional Heisenberg Lie algebra endowed with an inner product
$\sfQ.$  It is possible to choose an orthogonal basis $\calB =
\{\bfx_i\}_{i=1}^5$ such that all the Lie algebra relations are
determined by
\[ [\bfx_1,\bfx_2]= \bfx_5 \qquad \text{and} \qquad
[\bfx_3,\bfx_4]=\bfx_5.\] The two root vectors are $\bfy_{12}^5=
(1,1,0,0,-1)$ and $\bfy_{34}^5= (0,0,1,1,-1),$ and the Gram matrix $U$
is  $(\begin{smallmatrix} 3 & 1 \\ 1 & 3 \end{smallmatrix}).$    If
$\sfQ = \sum_{i=1}^5 q_i \, dx^i \otimes dx^i,$ the Ricci vector is 
\[ \bfRic = -\smallfrac{1}{2} \left( \frac{q_5}{q_1 q_2}(1,1,0,0,-1) +
\frac{q_5}{q_3 q_4}(0,0,1,1,-1)  \right).\] Setting $q_1= q_2 = \cdots
= q_5 = 1$ yields a soliton inner product with structure vector
$(1,1)^T$ and Ricci vector $\frac{1}{2}(-1,-1,-1,-1,2)$.
\end{example}

\subsection{The Ricci flow on the space of metric Lie algebras}
Isometries are preserved under the Ricci flow, so that as an initial
metric $g_0$ on a manifold $M$ evolves under the Ricci flow, the
metric $g_t$ at time $t$ has the isometry group $\Iso(g_0)$ of the
initial metric  as a subgroup of its isometry group $\Iso(g_t).$
Therefore, the  set of homogeneous metrics on $M$ is invariant under
the Ricci flow.

Let $(G,g)$ be an $n$-dimensional simply connected Lie group endowed
with left-invariant metric $g$ and  with corresponding metric Lie
algebra $(\frakg,\sfQ).$  Since the group $G$ acts simply transitively
on itself by isometries,  there exists an $G$-invariant global framing
of $G$.  The Ricci flow at any point   can therefore be expressed in
terms of the Ricci flow  at the tangent space to the identity.   We
view the Ricci flow as a flow $\phi_t$ on the space $\calP^+(\frakg)$
of  inner products on the vector space  $\frakg.$  We will call this
flow the {\em Ricci flow} for the  Lie  algebra $\frakg,$ and we will
write the solution for the initial condition $\sfQ_0$ as $\sfQ_t.$ 
 
The Ricci flow can be thought of as a flow on the space of metric Lie
algebras (modulo isometry).  We now describe the structure of that
space.  After fixing a basis for  a  Lie algebra $\frakg$ of dimension
$n,$ the Lie algebra can be identified with a  point  $\mu$ in
$\Lambda^2 V^\ast \otimes V,$  where $V$ is an $n$-dimensional vector
space.  We use  $\frakg_\mu$ to denote the algebra with underlying
space $V$ and its multiplication defined by $\mu.$  

A metric  Lie algebra $(\frakg,\sfQ)$ of dimension $n$ is a point  in 
\[ \calX_n =  \{ \mu \in \Lambda^2 V^\ast \otimes V  : \,
\text{$\frakg_\mu$ is a Lie algebra}\} \times \calP^+(V),   \] while a
metric nilpotent Lie algebra of dimension $n$ is a point in
\[ \calN_n =  \{ \mu \in \Lambda^2 V^\ast \otimes V  : \,
\text{$\frakg_\mu$ is a nilpotent Lie algebra}\} \times
\calP^+(V).   \] Note that the Jacobi identity and the nilpotency
conditions for a Lie algebra $\frakg_{\mu}$ are both polynomial
constraints on the structure constants so these sets are algebraic
subsets. 

Define an equivalence relation $\sim$ on $\calX_n$ and $\calN_n$ such
that  $(\frakg_\mu, \sfQ_1) \sim (\frakg_\nu, \sfQ_2)$ if and only if
$(\frakg_\mu, \sfQ_1)$ and  $(\frakg_\nu, \sfQ_2)$ are isometric.
Then  $ \widetilde{\calX_n}  = \calX_n   / \hskip -4pt \sim$
parametrizes the spaces of metric Lie algebras, and  $
{\widetilde{\calN_n}} =  \calN_n  / \hskip -4pt \sim $ parametrizes
the spaces of metric nilpotent Lie algebras.

E. Wilson showed that the simply connected nilmanifolds corresponding
to metric Lie algebras $(\frakn_\mu, \sfQ_1)$ and $(\frakn_\nu,
\sfQ_2)$ are isometric if and only if there is an isometric
isomorphism mapping one to the other (\cite{wilson82}). Hence, the
space $\widetilde{\calN_n}$ can be identified with the quotient space
for the natural action of $O(n)$ on $\calN_n.$  This action is the
restriction of the natural action of $GL_n(\boldR)$ on $\Lambda^2 V^\ast
\otimes V \times \calP^+(V),$  and is
defined as follows: for $g$ in $GL_n(\boldR),$
\[ g (\frakg_\mu, \sfQ ) =  (\frakg_{g \mu}, g \sfQ),\] where $g \mu
\in \Lambda^2 \frakg^\ast \otimes \frakg$ is given by
 \[ (g \mu) (\bfx,\bfy) = g \mu(g^{-1} \bfx, g^{-1} \bfy),\] and the
inner product $g\sfQ$ satisfies $g\sfQ(\bfx,\bfy) =
\sfQ(g\bfx,g\bfy).$

\subsection{The  projectivized Ricci flow on the space of
volume-normalized metric Lie algebras} Let $ \calP^+(\frakg) / \hskip
-4pt \sim$ denote the space  of  volume-normalized inner products on a
Lie algebra $\frakg,$ obtained through the equivalence relation $\sfQ
\sim \lambda \sfQ$ for $\sfQ$ in $\calP^+(\frakg)$ and $\lambda$ in
$\boldR^+.$ Denote the equivalence class of inner product $\sfQ$ by
$\overline{\sfQ}.$ The Ricci flow $\phi_t$ for the  Lie algebra
$\frakg$  projects to  a flow
\[ \overline{\phi}_t :  \calP^+(\frakg) / \hskip -4pt \sim  \,  \to
\calP^+(\frakg) / \hskip -4pt \sim\] on $\calP^+(\frakg)  / \hskip
-4pt \sim$ because $\phi_t(\lambda \sfQ) = \lambda \phi_t(\sfQ)$ for
any  $\sfQ$ in $\calP^+(\frakg)$ and any $\lambda > 0.$ We  call this
flow  the {\em projectivized Ricci flow} for $\frakg,$ and we will
write  $\overline{\phi}_t (\overline  \sfQ_0)$ as $\overline{\sfQ}_t,$
for $\overline \sfQ_0$ in  $ \calP^+(\frakg) / \hskip -4pt \sim.$ The
space $ \calP^+(\frakg) / \hskip -4pt \sim$ has a natural closure, the
compact  set  $\calP^{\ge 0}(\frakg)  / \hskip -4pt \sim,$ where
$\calP^{\ge 0}(\frakg)$ is the set of nontrivial  positive semidefinite
symmetric bilinear forms  on $\frakg$, and the equivalence relation
$\sim$ is defined as before.

We say that  a  metric Lie algebra $(\frakg,\sfQ)$   {\em collapses}
under the Ricci flow if  the limit   $  \overline  \sfQ_\infty =
\lim_{t \to \infty} \overline{\sfQ_t}$ exists in $ \calP^{\ge
0}(\frakg) / \hskip -4pt \sim$  but is in the boundary of  $
\calP^+(\frakg) / \hskip -4pt \sim$; that is, any representative
$\sfQ_\infty$ for the limiting normalized symmetric bilinear form $
\overline \sfQ_\infty$ is not positive definite.   We will then say
that $\sfQ$ and $\overline{\sfQ}$ {\em collapse to} $\overline
\sfQ_\infty.$

Let $\calX_n$ and $\calN_n$ be as defined previously.  Define a new
equivalence relation $\sim$ on $\calX_n \setminus (\{0\} \times
\calP^+(V))$ and $\calN_n \setminus (\{0\} \times \calP^+(V))$ so that
$(\frakg_\mu, \sfQ_1) \sim (\frakg_\nu, \sfQ_2)$ if and only if
$(\frakg_\mu, \sfQ_1)$ and $(\frakg_\nu, \sfQ_2)$ are homothetic. Let
\begin{align*} \overline{\calX_n}  &=  (\calX_n  \setminus (\{0\}
\times \calP^+(V))) / \hskip -4pt \sim, \\ \overline{\calN_n} &=
(\calN_n  \setminus (\{0\} \times \calP^+(V))) / \hskip -4pt \sim. 
\end{align*}  These spaces parametrize the families of nonabelian 
volume-normalized metric Lie
algebras and  nonabelian volume-normalized  
nilpotent metric Lie algebras, respectively.
We will represent the equivalence class of $(\frakg,\sfQ)$ in
$\overline{\calX_n}$ or $\overline{\calN_n}$   by
$\overline{(\frakg,\sfQ)}.$  Note that $(\frakg_{\mu},\sfQ) \sim
(\frakg_{\lambda \mu}, \sfQ),$ for $\mu \in  \Lambda^2 \frakg^\star
\otimes \frakg,$  $\lambda > 0,$ and $\sfQ$ in $ \calP^+(\boldR^n).$  

We extend the projectivized Ricci flow for a single Lie algebra to a
flow, also denoted by $\overline{\phi_t},$ on the spaces
$\overline{\calX_n}$ and  $\overline{\calN_n}.$  It is defined by 
\[\overline{\phi_t} \left( \overline{(\frakg, \sfQ)} \right) =
\overline{(\frakg,    \sfQ_t )},\] where $\sfQ_t$ is the solution
$\phi_t(\sfQ)$  to the Ricci flow  for the nilpotent Lie algebra
$\frakg$ with initial condition $\sfQ.$ The flow  is well-defined
since the Ricci flow commutes with homotheties.  For a metric Lie
algebra $(\frakg,\sfQ),$ $\overline{\phi_t}\left( \overline{(\frakg,
\sfQ)} \right)$ can be identified with $(\frakg,
\overline{\phi_t}(\overline{\sfQ})),$ where $\overline{\phi_t}$ is the
projectivized Ricci  flow for the single Lie algebra $\frakg.$ Note
the the algebraic structure is constant under the flow:   for all
finite $t\ge0,$ the  underlying nilpotent Lie algebras for
$\overline{(\frakg,\sfQ)}$ and for  $\psi_t(\overline{(\frakg,\sfQ)})$
are isomorphic. 

The next simple example illustrates how a metric Lie algebra can
collapse under the Ricci flow while having a nondegenerate limit point
for the projectivized Ricci flow.  

\begin{example}\label{h3-1} It was shown in \cite{isenberg-jackson-92}
that  the Ricci flow for the three-dimensional  Heisenberg Lie algebra
$\frakh_3$ is given by 
\[ [\sfQ_t]_{\calB} = \begin{bmatrix}  (3ct + 1)^{1/3} & 0 &  0 \\ 0 &
(3ct + 1)^{1/3} & 0 \\ 0  & 0 &   (3ct + 1)^{-1/3}\\
\end{bmatrix}, \] relative to a basis $\calB = \{ \bfx_1, \bfx_2,
\bfx_3 \}.$ The positive constant $c$ is the structure constant
$\alpha_{12}^3$ for  $\calB.$

 All inner products $\sfQ$ on $\frakh_3$  are homothetic, so   the
space $\overline{\calN_3}$ consists of a single point: the equivalence
class of  $(\frakh_3,\sfQ_0).$  At all finite times, the metric
nilpotent Lie algebras $(\frakh_3, \sfQ_t)$ that are solutions to the
Ricci flow with initial condition $\sfQ_0$ project to the same point
$\overline{(\frakh_3,\sfQ_0)}$ in $\calN_3.$  Therefore, the  point
$\overline{(\frakh_3,\sfQ_0)}$  is a fixed point for the projectivized
Ricci flow on $\overline{\calN_3}.$

Yet the inner product $\sfQ_0$  collapses under the flow, because
\[ \lim_{t \to \infty}  \frac{1}{ (3ct + 1)^{1/3}} 
\begin{bmatrix}  (3ct + 1)^{1/3} & 0 & 0 \\ 0  &   (3ct + 1)^{1/3} & 0
\\ 0  & 0 &   (3ct + 1)^{-1/3}\\
\end{bmatrix} =  \begin{bmatrix}  1 & 0 & 0  \\ 0 & 1  & 0 \\ 0 &  0
& 0 \\
\end{bmatrix}, \] forcing $\lim_{t \to \infty} \overline{\sfQ_t}$ to
be in  the boundary of  $ \calP^+(\frakh_3) / \hskip -4pt \sim.$
\end{example}

\subsection{The Lie bracket flow and the projectivized Lie bracket
flow}

 Let $(\frakg,\sfQ)$ be a metric Lie algebra, and let 
$\calB = \{ \bfx_i\}_{i=1}^n$ be an orthonormal
 basis for $\frakg,$ with
respect to which $\sfQ$ is written as $\sum_{i=1}^n dx^i \otimes dx^i.$ 
We will say that  $\calB$ is a {\em  stably Ricci-diagonal basis} if
$\calB$ is Ricci-diagonal for all metrics $\sfQ^\prime = \sum_{i=1}^n c_i 
dx^i \otimes dx^i,$ with $c_i > 0$ for $i = 1, \ldots, n.$
 It is often possible to find stably Ricci-diagonal bases. All
three-dimensional unimodular metric Lie algebras and many four-dimensional
metric Lie algebras have such bases (\cite{isenberg-jackson-92}, 
\cite{isenberg-jackson-lu}).   

If the basis $\calB$ is a stably Ricci-diagonal  basis
for $(\frakg,\sfQ)$, both the inner product and
the Ricci form remain diagonal with respect to $\calB$ under the Ricci
flow, and the positive functions $q_1, q_2, \ldots, q_n$ encode
the solution $\sfQ_t = \sum_{i=1}^n q_i \, dx^i \otimes dx^i$  to the
Ricci flow for $\frakg$ with initial condition  $\sfQ_0 = \sfQ.$   For
$1 \le j,k,l \le n$, define the function $ a_{jk}^l$ of $t$ by  
\begin{equation}\label{adef} a_{jk}^l  = \frac{q_l}{q_j q_k}
(\alpha_{jk}^l)^2 ,
\end{equation}  where $\alpha_{jk}^l$ is the structure constant
for $(\frakg, \sfQ_0)$ relative to the $\sfQ_0$-orthonormal basis
 $\calB$ as defined in Equation
\eqref{alpha-def}.   These functions represent structure constants
for $(\frakg,\sfQ_t)$  relative to the
$\sfQ_t$-orthonormal basis $\overline{\calB_{\sfQ_t}}$  defined in Section
\ref{structure}. 

 Let $a_1, \ldots,
a_m$ be an enumeration of  such functions  $ a_{jk}^l$ for $(j,k,l)$
in $\Lambda_\calB,$ in dictionary order.  Define the {\em structure
vector function} $\bfa : \boldR \to \boldR^m$  by 
\begin{equation}\label{def-of-a} \bfa_t = \bfa(t)  = (a_1(t), \ldots,
a_m(t))^T\end{equation} for each $t;$  this is just the $m \times 1$
structure vector relative to the basis  $\overline{\calB_{\sfQ_t}}$
for $(\frakg,  \sfQ_t).$  Observe that the vector $\bfa_t$ is positive at all
times $t \ge 0.$ We call the flow
 $\bfa_t$ the {\em Lie bracket flow}.

The structure vector function $\bfa_t$ defined in Equation
\eqref{def-of-a} takes values in $\boldR^m.$  We are interested in the
asymptotic behavior of $\bfa_t$ modulo rescaling; that is, the values
of $[\bfa_t]$ in projective space $P^{m-1}(\boldR).$ 
Since entries of $\bfa_t$ are
positive for all $t>0,$ we will describe the flow of $[\bfa_t]$
using the homogeneous variables
$s_1, \ldots, s_{m-1}$  given by
\begin{equation}\label{s-def} \bfs=  (s_1, \ldots, s_{m-1}) =
\left(\frac{a_1}{a_m}, \frac{a_1}{a_m}, \ldots,  \frac{a_{m-1}}{a_m}
\right).\end{equation}
These coordinates  parametrize the affine algebraic subset 
$\boldA_{m-1} = \{  (s_1 : s_2 : \cdots : s_{m-1}: 1) \}$ 
of  $P^{m-1}(\boldR).$ 

\subsection{Limits for the Ricci flow and the Lie bracket flow as time 
goes to infinity}

In general, if $[\bfa_t]$ converges to $[\bfa_\infty]$ under
 the projectivized Lie bracket flow (relative to
some stably Ricci-diagonal basis), the  Ricci flow does
not necessarily limit on a Lie group endowed with a left-invariant 
metric whose corresponding  metric Lie algebra has structure
constants equal to  $[\bfa_\infty].$
In general, caution must be 
exercised in extracting limits for the projectivized Ricci flow from limits for 
the projectivized Lie bracket flow.    
However, in the case of the projectivized Ricci flow on the space of
volume-normalized metric nilpotent Lie algebras, problems
 do not arise: one can say  that the  structure constants for the
  metric Lie algebra associated to the limiting homogeneous space 
are given by  $[\bfa_\infty].$ 

 It was demonstrated in \cite{wilson82} that
the isometry group of a simply connected
nilmanifold $N$ associated to metric nilpotent Lie algebra $(\frakn,\sfQ)$
is equal to the semidirect product $K \ltimes N$ 
of translations from $N$ and a compact isotropy group $K,$ and
the Lie algebra of $K$ is equal to $\Aut \left((\frakn, \sfQ)\right).$  
 Suppose that $\overline{(\frakn, \sfQ_t)}$ converges to
 $\overline{(\frakn_\infty, 
\sfQ_\infty)}.$    If $f \in \Aut\left((\frakn, \sfQ_t)\right)$ for all $t,$
 then $f \in \Aut\left((\frakn_\infty, \sfQ_\infty)\right),$  where
we identify $f$ with a linear map of the vector space $V$
used to define $\calN_n.$   
In contrast, whenever $\frakn \not \cong \frakn_\infty,$
the $n$-dimensional group of translational isometries for finite time does not 
coincide with the group of translational isometries for the limiting
nilmanifold. 

 There is a one-to-one correspondence between
 Lie brackets in  $\Lambda^2 V^\ast
\otimes V,$ modulo the action of $GL_n(\boldR),$  and
metric nilpotent Lie algebras $(\frakn,\sfQ)$ in $\calN_n.$  
The space $\calN_n$ is closed and invariant
for the flow $\phi_t,$ and if $\bfa_t$ is nilpotent for all $t,$ 
and $[\bfa_t] \to [\bfa_\infty],$ then $\bfa_\infty$ is nilpotent.  
Suppose that $\lim_{t \to \infty} [\bfa_t] = [\bfa_\infty].$  
Then there is a corresponding limit of metric nilpotent Lie algebras
$\lim_{t \to \infty} \overline{(\frakn,\sfQ_t)} = 
\overline{(\frakn_\infty,\sfQ_\infty)}$ in $\overline{\calN_n}.$  
At all finite times, the geometry of the nilmanifold class corresponding 
to $\overline{(\frakn,\sfQ_t)}$ is completely determined by left-multiplication
of the inner product $\sfQ_t$ over the Lie group $N = \exp(\frakn),$ while
 the geometry of the nilmanifold class corresponding 
to $\overline{(\frakn_\infty,\sfQ_\infty)}$ 
is completely determined by left-multiplication
of the inner product $\sfQ_\infty$ over the Lie group 
$N_\infty = \exp(\frakn_\infty).$
Hence, the local geometry (modulo homothety),
 of the nilmanifolds for  $\overline{(\frakn,\sfQ_t)}$
converges to the  local geometry  (modulo homothety) for the nilmanifold class 
defined by  $\overline{(\frakn_\infty,\sfQ_\infty)}.$ 

Observe that the vector $\bfa_t$ determines the metric nilpotent Lie algebra
 $(\frakn, \sfQ_t)$ up to  isometry
because it tells us the Lie algebra structure's structure constants 
relative to  an
orthonormal basis.   Hence  the  Lie bracket flow determines the Ricci flow
and the two flows are equivalent. 

Suppose that one could solve for $\bfa_t$ exactly.  
 Equation \eqref{Ricci vector} describes
the Ricci vector in terms of $\bfa_t$, which in turn gives  the Ricci
form $\ric_{\sfQ_t}$ as a function of time.  Using the  equation for
the Ricci flow,  one can solve for the functions  $q_1, \ldots, q_m$
giving the inner product $\sfQ_t$ as a function of time by integrating
$-2 \ric_{\sfQ_t}$ with respect to $t.$   Also note that since the
connection and curvatures for $(\frakn, \sfQ_t)$ depend only on the
structure constants relative to an orthonormal basis, these geometric
quantities are determined by the Lie bracket flow.   In order to
compute or estimate geometric quantities, it is sufficient to consider 
only  the
Lie bracket flow, which in many cases is easier to work with than the
Ricci flow.

\section{Systems of ordinary differential equations  } \label{systems}

In this section, we  set up  systems of ODE's for the Ricci flow and
for the projectivized bracket flow.

\subsection{O.D.E.'s for the Ricci flow and the Lie bracket
flow}\label{ODEs}

The next theorem describes how the structure vector function $\bfa_t$
evolves under the  Ricci flow, assuming the existence of a stably
Ricci-diagonal basis.
\begin{theorem}\label{ricci-system}  Let $(\frakg,\sfQ)$ be a metric
Lie algebra with stably Ricci-diagonal basis $\calB = \{\bm{x}_i\}_{i =
1}^n.$  Let  $Y$ be the root matrix and  let
$U$ be the Gram matrix for $(\frakn,\sfQ)$ relative to $\calB.$
   Let  $\sfQ_t$ be the solution to the Ricci flow for $\frakg$
with initial condition $\sfQ_0=\sfQ,$  and  let the functions $q_1,
\ldots, q_n$ be defined by $\sfQ_t = \sum_{i=1}^n q_i \,dx^i \otimes
dx^i$.  Let $\bfa_t = (a_i(t))$ be the structure vector function for
$(\frakg,\sfQ_t)$ as defined in Equation \eqref{def-of-a}.  Then 
\begin{equation}\label{ricci-flow-gen}  
\left(\frac{q_1^\prime}{q_1},
\ldots, \frac{q_n^\prime}{q_n} \right)  =  -2 \,
\bfRic_\calB,\end{equation} and
\begin{equation}\label{Z-gen} 
 \left( \frac{a_1^\prime}{a_1},  \frac{a_2^\prime}{a_2},  \ldots,
\frac{a_m^\prime}{a_m}\right) = 2 \, \bfRic_\calB  Y^T. 
\end{equation} In the case that $\frakg$ is nonabelian and nilpotent, 
\begin{equation}\label{ricci-flow} 
 \left(\frac{q_1^\prime}{q_1},
\ldots, \frac{q_n^\prime}{q_n} \right)  =  \bfa^T Y,\end{equation} 
and
\begin{equation}\label{Z} 
 \left( \frac{a_1^\prime}{a_1},  \frac{a_2^\prime}{a_2},  \ldots,
\frac{a_m^\prime}{a_m}\right) = - \bfa^T U.
\end{equation} 
\end{theorem} 

\begin{remark}
Observe that the O.D.E.'s for $a_1, a_2, \ldots, a_m$ are quadratic,
and they can always be put in log-linear form: for example, 
Equation \eqref{Z-gen} becomes
 $(\ln a_i)^\prime = -\sum_{j=1}^m u_{ij} a_j ,$ for $i =1 ,\ldots, m.$  
\end{remark}

\begin{remark}
The functions $a_i^\prime$ are smooth and locally bounded on 
$\boldR^m;$ hence solutions for $a_i(t)$ exist for all $t > 0.$
By the same reasoning, solutions to $q_i(t)$ exist for all time.
However they are not guaranteed to be positive for all time: for
example, the round metric on $SU(2)$ becomes zero in finite
time (\cite{isenberg-jackson-92}).
\end{remark}

\begin{proof}  With respect to the basis $\calB,$ the inner product is
represented by the diagonal matrix
\[  [\sfQ]_\calB = \diag(q_1, q_2, \ldots, q_n),\] and the Ricci form
is represented by a diagonal matrix $[\ric]_\calB.$ If we rewrite  the
Ricci form with respect to the orthonormal basis  $\overline{\calB}$
obtained by rescaling $\calB,$ we get
\[ [\ric]_\calB  = [\sfQ]_\calB [\ric]_{\overline \calB} .\] By
definition of the Ricci vector, the matrix  $[\ric]_{\overline \calB}$
is the diagonal matrix whose diagonal entries are the entries  of the
Ricci vector $\bfRic_{\calB}$ as defined in Equation
\eqref{def-of-Ric}.  By equating the diagonal entries of the matrices
on both sides of the matrix equation $ [\sfQ_t]_\calB^\prime = -2[\ric_{\sfQ_t}]_\calB $,  
the Ricci flow may then be written 
\begin{equation}\label{q flow} \left (\frac{q_1^\prime}{q_1}, \ldots,
\frac{q_n^\prime}{q_n} \right)  =  -2 \, \bfRic_{\calB}
.\end{equation}

Changing variables to $\ln q_i$, for $i=1, \ldots, m,$ yields
\begin{equation}\label{flow-equation} (\ln q_1, \ldots, \ln
q_n)^\prime =  -2 \, \bfRic_\calB  .\end{equation} Now we compute the
derivative of the functions 
$\ln a_i,$ for $i = 1, \ldots, m,$ defined in Equation \eqref{adef},
temporarily switching
to the indexing $a_{jk}^l, (i,j,k) \in \Lambda_\calB,$ for the
functions $a_i, i =1, \ldots, m.$ For $(i,j,k) \in \Lambda_\calB$, the
derivative of  $\ln(a_{jk}^l)$  is 
\begin{align*}  
\left( \ln (a_{jk}^l)  \right)^\prime 
&= \left( \ln
\left( (\alpha_{jk}^l)^2 \frac{q_l}{q_jq_k} \right) \right)^\prime \\
 &=   -(\ln q_j)^\prime - (\ln q_k)^\prime +
(\ln q_l)^\prime \\ &= - \left( \ln q_1, \ldots, \ln q_n\right)^\prime
\cdot \bfy_{jk}^l. 
\end{align*} 
Using Equation \eqref{flow-equation} to rewrite the right
side of the previous line, we get
\[ \left( \ln (a_{jk}^l)  \right)^\prime =  2  \, \bfRic_\calB
(\bfy_{jk}^l)^T \] These $m$ linear equations conjoin to become
Equation \eqref{Z-gen}.  If $\frakg$
is nilpotent,  by Equation \eqref{Ricci vector} of  Theorem
\ref{riccitensor},  $\bfRic_\calB = -\smallfrac{1}{2} \bfa^T Y.$  
Substitution of this into Equations \eqref{ricci-flow-gen} and \eqref{Z-gen}
and the identity $YY^T = U$ yield
Equations \eqref{ricci-flow} and \eqref{Z}.
\end{proof} To illustrate Theorem \ref{ricci-system}, we revisit some
of the previous examples.  
\begin{example} For $(\frakh_3,\sfQ)$ as in Example \ref{h3-a}, if we
let $a_1 = \frac{q_3}{q_1q_2},$ Equation \eqref{Z} reduces to the single
equation
\[ a_1^\prime = -3a_1^2 .\]  Integrating, we get $a_1(t) = (3t +
c)^{-1},$  where $1/c$ is the structure constant $\alpha_{12}^3$ for
the initial orthonormal basis.  Then   
Equation \eqref{ricci-flow} is
\[ (\ln q_1, \ln q_2, \ln q_3)^\prime = (3t + c)^{-1} (1,1,-1),\]
and integration gives
the solutions $q_1, q_2, q_3$ already presented in Example \ref{h3-1}.

For  $(\frakl_4,\sfQ)$ as in Example \ref{l4-a}, when  we let $a_1 =
\frac{q_3}{q_1q_2}$ and  $a_2 =   \frac{q_4}{q_1q_3},$ the system of
equations from Equation \eqref{Z} is
\[  a_1^\prime = -3a_1^2, \qquad   a_2^\prime = -3a_2^2 ,\] which
decouples, so it is not hard to solve for

$a_1$ and $a_2,$ and then $q_1, q_2, q_3$ and $q_4.$
For the metric Heisenberg algebra $(\frakh_5,\sfQ)$ as in Example
\ref{h5-a}, after letting $a_1 =  \frac{q_5}{q_1q_2}$ and  $a_2 =
\frac{q_5}{q_3q_4},$ we get the system
\begin{align}\label{h5-odes} a_1^\prime &= -3a_1^2 - a_1a_2 \\
a_2^\prime &= -a_1a_2  -3a_2^2 , \notag \end{align} which, in
contrast to the first two examples, has no simple explicit general
solution.  Instead, it is necessary to make a qualitative analysis of
a projectivization of such a system; this is the goal of the next
section.
\end{example}

\subsection{O.D.E.'s for the projectivized Lie bracket
flow}\label{odes-hom}

Now we define some matrices and hyperplanes that are needed to state
the main theorem of the section.  Define the $(m-1) \times m$ matrix
$P$ by
\[ P = \begin{bmatrix}   1 & 0 & 0 &  \cdots & 0 & -1 \\  0 &  1 & 0 &
\cdots & 0 &-1 \\ 0 & 0  & 1  &           & 0 & -1 \\ \vdots & \vdots
& &  \ddots &\vdots & \vdots \\  0  &  0 & 0  &\cdots &1 & -1\\
\end{bmatrix}.\] Let $U$ be an $m \times m$  matrix.   For $i = 1,
\ldots,m-1,$ define  the vector $\bfn_i$ by  $\bfn_i = \bfe_i^T
PU,$ where $\{\bfe_i\}_{i=1}^{m-1}$ is the standard orthonormal
basis for $\boldR^{m}.$  The vector  $\bfn_i$ is the $i$th row of  the
matrix $U$ minus the $m$th row of the matrix $U.$

The vectors $\bfn_1, \ldots, \bfn_{m-1}$ define a set of $m-1$
hyperplanes  $\calH_{1}^0, \calH_{2}^0, \ldots, \calH_{m-1}^0,$  where
we let
\begin{equation}\label{hyperplane definition} \calH_i^0 =
\bfn_i^\perp, \quad \text{for $i=1, \ldots, m-1$.}\end{equation} For
$i=1,  \ldots, m-1,$  define the open half-spaces  $\calH_{i}^+$ and
$\calH_{i}^-$ by 
\begin{align*}  \calH_{i}^+ &= \{ \bfa \in \boldR^m\, | \, \bfa \cdot
\bfn_i  > 0 \}\\ \calH_{i}^- &=  \{ \bfa \in \boldR^m \, | \, \bfa
\cdot \bfn_i  < 0 \}.
\end{align*}

The next proposition describes some properties of these hyperplanes and
their normal vectors, when $U$ is the Gram matrix for  a metric
nilpotent Lie algebra $(\frakn,\sfQ)$ relative to an orthogonal basis
$\calB.$ 
\begin{prop}\label{PU}  Let $(\frakn,\sfQ)$ be a nonabelian  metric
nilpotent Lie algebra with orthogonal  Ricci-diagonal basis $\calB.$
If $U$ is the Gram matrix for $(\frakn,\sfQ)$ with respect to
$\calB,$ then the intersection $\cap_{i=1}^{m-1} \calH_i^0$ of the
hyperplanes $\calH_{1}^0, \calH_{2}^0, \ldots, \calH_{m}^0$ is equal
to
\[   \ker PU = \{ \bfv \, : \, U \bfv = \lambda \onevector{m} \quad
\text{for some $\lambda$ in $\boldR$}\}.
 \] If in addition,  $\alpha_{jk}^j \ne 0$ for all $1 \le j,k \le n,$
then  for all $i=1, \ldots, m-1,$ 
\begin{enumerate}
\item{ the $i$th entry of the vector $\bfn_i$ is positive and the
$m$th entry of  $\bfn_i$ is negative, and} \label{neg-m-m}
\item{ the point $(0,0,\ldots,0,1)$ lies in the open half space
$\calH_i^-.$ }\label{0 is neg}
\end{enumerate}
\end{prop} Note that the subspace  $\ker PU$ of $\boldR^m$ is always
nontrivial  because  $\rank PU = m-1.$
\begin{proof}  By the definitions of the vectors $\bfn_i$ and the
hyperplanes $\calH_i^0,$ a vector $\bfv$ lies in $\ker PU$ if  and
only if $\bfv$ is in  $\calH_i^0$ for all $i = 1, \ldots, m-1.$ For a
vector $\bfv$ in $\boldR^m,$ $PU\bfv = \bfzero$ if and only if the
$i$th entry and the $m$th entry  of the vector  $U\bfv$ are the same
real number $\lambda,$ for $i=1, \ldots, m-1.$ This is true if and
only if $U\bfv = \lambda \onevector{m}.$ Therefore, the kernel of the
matrix $PU$  is spanned by vectors $\bm{v}$ so that $U\bm{v}$ is a
scalar multiple of $[1].$ This proves the first part of the proposition.

To prove Statements \ref{neg-m-m} and \ref{0 is neg},
 simply use the definition of $PU$ and that when
$\alpha_{jk}^j \ne 0$ for all $j$ and $k,$ the diagonal entries of $U
= (u_{ij})$ are all three, while the off-diagonal  entries are all in
the set $\{-2,-1,0,1,2\}.$  Then 
\[ (0,0,\ldots,0,1) \cdot \bfn_i = (\bfn_i)_m = u_{im} - 3 <  0 \] for
$i=1, \ldots,m-1.$ 
\end{proof}

 The next theorem describes the evolution of the structure vector
$\bfa_t$  for $(\frakn,\sfQ_t),$ modulo rescaling, as the metric
nilpotent  Lie algebra $(\frakn,\sfQ)$ evolves under the Ricci flow.
The hyperplanes $\calH_{1}^0, \calH_{2}^0, \ldots, \calH_{m-1}^0$
defined in Equation \eqref{hyperplane definition} are homogeneous sets
in $\boldR^m.$  For all $i = 1, \ldots, m,$  $\calH_i^0 \setminus
\{\bfzero\}$   projects to an algebraic set $[\calH_i^0]$ of
codimension one in $P^{m-1}(\boldR)$ whose intersection with
$\boldA_{m-1}$ is a hyperplane.  For $i=1, \ldots, m-1,$  define the
functions $\eta_i : \boldR^{m-1} \to \boldR$ by 
\begin{equation}\label{eta-def} \eta_i(\bfs) = \bfn_i \cdot
(s_1,\ldots,s_{m-1},1),
\end{equation} for $\bfs$  in  $\boldR^{m-1}.$ 

\begin{theorem}\label{projectivized} Let $(\frakn,\sfQ)$ be a metric
nonabelian nilpotent  Lie algebra with  basis $\calB =
\{\bm{x}_i\}_{i = 1}^n$ that is stably Ricci-diagonal.  
Let $m = |\Lambda_\calB|,$ and let $\sfQ_t =
\sum_{i=1}^n q_i \, dx^i \otimes dx^i$ denote the solution to the
Ricci flow at time $t,$ written relative to $\calB.$ Let  the
functions $a_1,\ldots, a_m$ and  $s_1, \ldots, s_{m-1}$ be as defined
in Equations \eqref{adef} and \eqref{s-def} respectively.  Then
\begin{equation}\label{slopes} s_i^\prime (t) =  - a_m s_i
\eta_i(\bfs(t)) .
\end{equation} for $i=1, \ldots, m-1.$
\end{theorem}

\begin{proof}[Proof of Theorem \ref{projectivized}] 
Using the quotient rule, we get 
\[ s_i^\prime =  \left( \frac{a_i}{a_m} \right)^\prime =
\frac{a_i^\prime a_m - a_m^\prime a_i}{a_m^2}. \] By  Theorem
\ref{ricci-system}, for $i =1,\ldots, m-1,$ the function $a_i^\prime$ is
equal to $-a_i$ times the product of the  $i$th row of $U$ with $\bfa,$
so 
\begin{align*}  s_i^\prime  &=   \frac{a_i}{a_m} \left[ -(\text{row
$i$ of $U$}) \cdot \bfa +  (\text{row $m$ of $U$}) \cdot \bfa \right]
\\ &= -s_i  \, \left[ (\text{row $i$ of $PU$}) \cdot \bfa \right]  \\
&= -s_i a_m  \, \left[ (\text{row $i$ of $PU$}) \cdot (s_1, \ldots,
s_{m-1},1)  \right] 
\end{align*} Thus,  $s_i^\prime (t) =  - a_m s_i \eta_i(\bfs(t)),$ as
desired.
\end{proof} 

The hyperplanes $[\calH_i^0], i =1, \ldots, m$ divide the affine space
$\boldA_{m-1}$ into chambers.  The theorem implies that the
general direction of a $\bfs_t$ trajectory depends only upon which
chamber a point $\bfs$ in $\boldA_{m-1}$ is on, and that equilibrium
points come from soliton metrics (solutions to $Uv = [1]$).  

\begin{remark}\label{time change} Orbits of the system in Equation
\eqref{slopes} agree with those of the system 
\begin{equation}\label{simplersystem}  (\ln s_i)^\prime  =  -
\eta_i(\bfs).
\end{equation} The direction of the trajectories for the two systems
is the same  because $a_m > 0.$   Although the system in Equation
\eqref{slopes} may be quite difficult to solve exactly, as the
functions $\eta_i, i = 1, \ldots, m$ are linear in the variables $s_i$, the
system in Equation \eqref{simplersystem} is more tractable. 
\end{remark}

 In the next example,  we illustrate the definitions of the vectors
$\bfn_i$ and hyperplanes $\calH_i^0,$ and we give an application of
the previous theorem.  
 
\begin{example}\label{prototype} Let $(\frakn,\sfQ)$ be the
five-dimensional metric Lie algebra that with respect to an orthogonal
basis $\calB = \{ \bfx_i\}_{i=1}^5$ has the following bracket
relations:
\[ [\bfx_{1},\bfx_{3}]=\bfx_{4},  \qquad [\bfx_{1},\bfx_{4}]=\bfx_{5},
\qquad [\bfx_{2},\bfx_{3}]=\bfx_{5}   .\]  By Theorem
\ref{riccitensor}, $\calB$ is a Ricci diagonal basis.  By Remark
\ref{good}, the Ricci form is diagonal for all rescalings of vectors
in $\calB,$ so the basis is stably Ricci-diagonal.  Then 
\[ \Lambda_\calB = \{(1,3,4), (1,4,5), (2,3,5)\},\] and the Gram
matrix $U$ and the matrix $PU$ are given by  
\[ U =\begin{bmatrix} 3 & 0 & 1 \\ 0 & 3 & 1 \\ 1 & 1 &
3 \end{bmatrix}  \qquad \text{and} \qquad  PU =\begin{bmatrix} 2 & -1
& -2 \\ -1 & 2 & -2  \end{bmatrix}. \] The kernel of $PU$ is spanned
by $\bfv = (2,2,1)^T;$ note that  $U\bfv = 7 (1,1,1)^T.$ Therefore, by
Theorem \ref{Ua},  if the inner product $\sfQ$ which has $\calB$ as an
orthogonal basis has  structure vector $\bfa = (\frac{q_4}{q_1q_3},
\frac{q_5}{q_1q_4},\frac{q_5}{q_2q_3})$ equal to $(2,2,1)^T,$ then it
is  soliton with soliton constant $\beta = -7/2.$  In that case the
Ricci vector is   $-\smallfrac{1}{2}(4,1,3,0,-3),$ and the  derivation
$D = \Ric - \beta \Id$ is represented by  $[D]_\calB,$ the diagonal
matrix $\diag(3,6,4,7,10).$  This example will be revisited later in
Example \ref{prototype-2}, and soliton inner products  will be
described.

Here, $\bfn_1 = (2,-1,-2)$ and $\bfn_2 = (-1,2,-2).$ The hyperplanes
$\calH_1^0$ and  $\calH_1^0$ are the planes  $\bfn_1^\perp = (2a_1-
a_2 -  2a_3 = 0) $ and $\bfn_2^\perp = (-a_1 +2a_2 - 2a_3 = 0)$ in
$\boldR^3.$ The sets $\calH_1^0 \cap (a_3 \ne 0)$ and $\calH^0_2 \cap
(a_3 \ne 0)$ in $\boldR^3$ project  to  lines $l_1 = [\calH_1^0]$ and
$l_2 = [\calH_2^0]$ in the subset $\boldA_2$ of $P^2(\boldR).$   In
$s_1$-$s_2$ coordinates on $\boldA_2,$ the lines $l_1$ and $l_2 $ are
given by  $2s_1 -  s_2 = 2 $ and $s_1 - 2 s_2 =  -2 $ respectively,
and they intersect in the point $(s_1,s_2) = (2,2).$  Therefore the
lines $l_1$ and $l_2$ intersect in the point  $(2:2:1)$ in  $\boldA_2
\subset P^2(\boldR),$ whose preimage under the map $\bfv \mapsto
[\bfv]$ is the line of intersection $\boldR (2,2,1) $ of the planes
$\calH_1^0$ and $\calH_2^0$ in $\boldR^3.$  This line contains the
orbit of the structure vector $\bfa_t$ for any soliton metric under the
Ricci flow. See Figure 1.

The linear functions
\begin{align*} \eta_1(s_1,s_2) &= 2s_1 - s_2 - 2 \quad \text{and}\\
\eta_2(s_1,s_2) &= -s_1 + 2s_2 - 2 
\end{align*} determine coordinates on $\boldA_2$ such that
$\eta_1\left( (0,0) \right) < 0$ and $\eta_2\left( (0,0) \right)  <
0$, and at the point of intersection  $(s_1,s_2) = (2,2)$ of the lines
$l_1$ and $l_2,$ $(\eta_1,\eta_2) = (0,0).$ 

The main features of the slope field for the flow $\bfs_t$ are that
\begin{itemize}
\item{Points above $l_2$ move downward,  points below $l_2$ move upward, and
points on $l_2$ other than $(2,2)$ move horizontally; and}
\item{Points to the left of 
$l_1$ move left,  points to the right of $l_1$ move to the right, and
points on $l_1$ other than $(2,2)$ move vertically.}
\end{itemize} 
(Compare to the arrows drawn on the axes, and Part \ref{0 is neg} of Proposition \ref{PU}.) 
 It can be seen from the
slope field that for any initial point $\bfs_0 > 0,$ eventually the trajectory
enters the cone $\eta_1(\bfs) \eta_2(\bfs) > 0,$ and once inside that
cone, asymptotically approaches $(2,2).$

We note that the fixed points on the axes correspond to soliton metric
nilpotent Lie algebras whose underlying algebra is not isomorphic to $\frakn.$

\begin{figure}\label{figure-prototype}
\caption{The flow $\bfs(t)$ for Example \ref{prototype}}
\setlength{\unitlength}{.8mm}
\begin{picture}(45,50)(0,0)  \thinlines  \put(20,10){\line(1,2){16}}
\put(28,18){$l_1$}  \put(10,20){\line(2,1){35}} \put(15,28){$l_2$}
\thicklines  \put(10,0){\line(0,1){45}}  \put(0,10){\line(1,0){45}}
\put(10,47){\vector(0,1){0.6}}  \put(45,10){\vector(1,0){0.6}}
\put(4,42){$s_2$} \put(47,10){$s_1$}  \put(10,10){\circle*{2}}
\put(10,20){\circle*{2}} \put(20,10){\circle*{2}}
\put(30,30){\circle*{2}} \put(10,16){\vector(0,1){0.8}}
\put(16,10){\vector(1,0){0.8}} \put(10,30){\vector(0,-1){0.8}}
\put(30,10){\vector(-1,0){0.8}}
\end{picture}
\end{figure}

\end{example}

\subsection{Invariants for the Ricci flow}  In previous studies of the
Ricci flow for homogeneous spaces,  the systems of ODE's for the Ricci
flow are solved by finding invariant quantities under the flow.   The
next proposition gives an easy way to find  some of these invariant
quantities:  vectors in the kernel of the root matrix $Y$  for
$(\frakn,\sfQ)$ relative to orthogonal Ricci-diagonal basis  $\calB$
yield conserved monomial quantities for the  Ricci flow for $\frakn.$ 

\begin{prop}\label{conserved} Let $(\frakn,\sfQ)$ be a nonabelian
metric nilpotent Lie algebra with  orthogonal Ricci-diagonal basis
$\calB.$  Suppose that  there is a stably Ricci-diagonal basis with respect to
which the solution
 $\sfQ_t$ to the Ricci flow
is expressed as $\sfQ_t = \sum_{i=1}^n q_i \,dx^i \otimes
dx^i$.   Let  $Y$ be the root matrix for $(\frakn,\sfQ)$
relative to $\calB.$  Let $\bfa_t$ be as defined in Equation
\eqref{adef}.

 The  constant vector $\bm{d}= (d_1, \ldots, d_n)$ satisfies the
condition that  $ \bm{d}  Y^T \bfa_t = \bm{0}$ for all $t$ if and only
if  the quantity  $q_1^{d_1} q_2^{d_2} \cdots q_n^{d_n}$ is preserved
under the Ricci flow.  In particular, if  $ \bm{d} Y^T = \bm{0}$ then
$q_1^{d_1} q_2^{d_2} \cdots q_n^{d_n}$ is preserved. 
\end{prop}

\begin{proof} The quantity $q_1^{d_1} q_2^{d_2} \cdots q_n^{d_n}$ is
preserved if and only if its natural logarithm   $\sum_{i = 1}^n d_i
\ln q_i$  is preserved.  This is true if and only if for all $t>0,$
\begin{align*} 0 &= \frac{d \,}{dt}  \sum_{i = 1}^n d_i  (\ln q_i) \\
& = \sum_{i = 1}^n d_i (\ln q_i)^\prime\\ & = (d_1, \ldots, d_n)
{(\ln q_1, \ldots, \ln q_n)^\prime}^T.
 \end{align*} Rewriting $(\ln q_1, \ldots, \ln q_n)^\prime$ using
Equation \eqref{ricci-flow},  we see that  $q_1^{d_1} q_2^{d_2} \cdots
q_n^{d_n}$ is preserved if and only if
\[  ( d_1, \ldots,  d_n)   Y^T  \bfa  = 0 . \] Clearly,  if  $ \bm{d}
Y^T = \bm{0}$ then  $q_1^{d_1} q_2^{d_2} \cdots q_n^{d_n}$ is
preserved. 
\end{proof}

\begin{example} For $(\frakh_3,\sfQ)$ as in Example \ref{h3-a}, 
conserved quantities $q_1q_3$, $q_2q_3$ and $q_1/q_2$  for the Ricci
flow for $(\frakh_3,\sfQ)$ come from the vectors $(1,0,1)^T, (0,1,1)^T$
and $(1,-1,0)^T$  in the kernel of the root matrix $Y = (1,1,-1)$.
\end{example}

\subsection{The phase portrait for the
projectivized Lie bracket flow}  The
structure vector $\bfa_t$ for   a metric nilpotent Lie algebra
 $(\frakn,\sfQ)$ evolves under the Ricci
flow according to the law in Equation \eqref{Z}.  The projection of
the $\bfa_t$ flow on $\boldR^m$ to the flow $[\bfa_t]$ on
$P^{m-1}(\boldR),$ as represented by the coordinates $\bfs_t$ in
$\boldA_{m-1},$ obeys  Equation \eqref{slopes}; and by Remark
\ref{time change},  has the same trajectories  as the solutions to
Equation \eqref{simplersystem}.  We describe the properties
of a flow satisfying Equation  \eqref{simplersystem}.

\begin{lemma}\label{limit of a} Let $U$ be the Gram matrix
for   a nonabelian nilpotent metric Lie algebra $(\frakn,\sfQ)$ with
respect to  an orthogonal Ricci-diagonal basis $\calB,$ with  $|
\Lambda_\calB | = m.$ 
Let the functions $\eta_i$ 
and hyperplanes $[\calH_i^0],$ for $i=1, \ldots, m-1,$ be as defined in
Section \ref{systems}.
Then 
 the system of ordinary differential equations 
\begin{equation*}
 (\ln s_i)^\prime =  -  \eta_i(\bfs_t), \quad i = 1, \ldots, m-1, 
 \end{equation*}
has the following properties.
\begin{enumerate}
\item{The sets $\partial(\bfs > 0)$ and $(\bfs > 0)$ are invariant under the 
flow.  }\label{inv sets}
\item{ The
set of equilibrium solutions in $(\bfs \ge 0)$
 is nonempty and compact, 
and is equal to the union $\calS = \cup \calS_M$
 of all sets of the form 
\[ \calS_M = \bigcap_{i \not \in M} [\calH_i^0] 
\cap \bigcap_{i  \in M} (s_i = 0) \cap (\bfs \ge 0), \]
where $M$ varies over all subsets of $\{1, \ldots, m-1\}.$}
\item{Define the subset $\calS^-$ of $\calS$
 to be the union of the points $\bfb = (b_i)$ in
$ \calS$ such that there exists $i$ such that
$b_i=0$ and $\eta_i(\bfb) <  0.$  All points in  $\calS^-$
repel nearby  points in $(\bfs > 0).$
}
 \label{fixed pts}
 \end{enumerate}
 \end{lemma}
We remark that it is possible to have continuous families of fixed points
in  $(\bfs > 0)$
 for the system in Equation \eqref{simplersystem}
 (see Examples 
27 and 28 and Theorem 29 from \cite{payne05}).
\begin{proof}
The vector function $\bfs_t$ evolves according to law 
$\bfs^\prime = \bfF(\bfs),$ where the  vector field 
 $\bfF = (F_i)$ is defined by $F_i(\bfs) = - s_i \eta_i(\bfs)$
for $i = 1, \ldots, m.$
Clearly $\bfF$ is tangent to the boundary of $(\bfs > 0),$
so that the boundary is an invariant set. Since solutions exist
everywhere and are 
unique, the interior $(\bfs > 0)$
 is forced to be invariant under the flow also.   
This proves Part \ref{inv sets}.  

 It is immediate from the definition of $\bfF$ that  
 the solution $\bfs_t$ with initial condition 
 $\bfs(0) \ge 0$ 
 is an equilibrium solution if and only if
for $i = 1, \ldots, m-1,$ either $\eta_i(\bfs_0) = 0$ or 
$s_i = 0,$ so the set of equilibrium points equals 
the set $\cup \calS_M.$  

We need to show compactness of the set $\calS.$  Clearly it
is closed.   We will show that $\calS_M$ has compact
closure for all subsets $M$ of $\{1, \ldots, m-1\}.$
Consider the restriction of the flow to the set 
\[ A_M = \{ \bfs = (s_i) 
\, : \text{$s_i=0$ for $i \in M,$ $s_i > 0$ for $i \not \in M$} \}. \]
The fixed points in this satisfy $s_i = 0$ for $i \in M,$
and for $\eta_i(\bfs_t) = 0$ for $i \not \in M.$
  After renaming $s_i, i \not \in M,$ as $t_1, \ldots, t_k$ for some $k,$
these can be written as $PU^\prime \bft = \bfzero$ for the $k \times k$
Gram matrix that is a minor of $U$ obtained by crossing out rows
and columns with indices in $M.$  By Proposition \ref{PU}, solutions 
satisfy $U \bft = -2\beta^\prime [1]$ for some $\beta < 0.$  By Theorem 2 of
\cite{payne05}, the set of such $\bft$ in $A_M$ is compact.

Suppose that $\bfb = (b_1, \ldots, b_{m-1})$ is in 
the boundary of $(\bfs > 0)$ and is
in $\calS^-.$  Then
 there exists an $i$ so that 
 $b_i = 0$ and $\eta_i(\bfb) < 0.$  Because $s_i^\prime = -s_i \eta_i(\bfs),$
there exists $\epsilon > 0$ so that 
 $s_i^\prime > 0$ for all points in the open subset $V = \{ \bfs > 0 :  
\| \bfs - \bfb \| <   \epsilon \}$ of $(\bfs > 0).$
If $\bfs$ is in  $V,$ then $0< s_i < \epsilon$ and 
$s_i^\prime $ is positive and bounded below, so $\bfs$ leaves $V$
 in finite time.  Therefore,  $\bfb$
repels all nearby points in $(\bfs > 0).$
\end{proof}

The next example illustrates  the phase portrait for the
projectivized Lie bracket flow  when there is no positive
solution $\bfv$ to $U\bfv = [1].$   To our knowledge, the smallest Gram
matrix $U^\prime$ arising from a metric nilpotent Lie algebra is $8 \times 8$ 
(see Example \ref{R6}).  For the sake of 
simplicity, in our example we have chosen 
a $3 \times 3$ matrix $U$ that
does not necessarily come from a nilpotent Lie algebra, but for which
the  $\bfa_t$  flow defined by $U$ and Equation \eqref{Z} 
has the same   qualitative features as
for the $\bfa_t$ flow defined by  $U^\prime.$
\begin{example}\label{l4-b} Suppose that
\[ U = \begin{bmatrix} 3 & 2 & 0 \\ 2 & 3  & 2 \\  0 & 2 &
3 \end{bmatrix}. \qquad \text{Then} \qquad  PU =\begin{bmatrix} 3 & 0 &
-3 \\ 2 & 1 & -1  \end{bmatrix}, \] 
the vectors $\bfn_1$ and $\bfn_2$
are  $(3,0,-3)$ and $(2,1,-1),$ respectively, and $\ker PU$ is spanned
by $\bfv = (1,-1,1).$  This  vector $\bfv$ is  the solution to
$U\bfv = \onevector{3}.$ 
  The
hyperplanes $\calH_1^0$ and $\calH_2^0$ are given by $(a_1 = a_3)$ and
$(2a_1 + a_2 - a_3 = 0)$ respectively, and  the functions $\eta_1$ and
$\eta_2$ are given by
\begin{align*} \eta_1\left((s_1,s_2)\right) &= 3s_1 - 3 \\ \eta_2
\left((s_1,s_2)\right) &= 2s_1 + s_2 - s_3. 
\end{align*} The hyperplanes intersect in the line $\boldR\bfv.$ 

\begin{figure}
\caption{The flow $\bfs(t)$ for Example \ref{l4-b}}
\setlength{\unitlength}{.8mm}
\begin{picture}(70,75)(0,0) \thinlines \put(30,0){\line(0,1){70}}
\put(33,60){$l_1$} \put(7,56){\line(1,-2){28}} \put(3,59){$l_2$}
\thicklines \put(10,0){\line(0,1){70}} \put(0,30){\line(1,0){70}}
\put(10,70){\vector(0,1){0.6}} \put(70,30){\vector(1,0){0.6}}
\put(12,65){$s_2$} \put(63,33){$s_1$} \put(10,50){\circle*{2}}
\put(10,45){\vector(0,1){0.8}} \put(10,55){\vector(0,-1){0.8}}
\put(30,30){\circle*{2}} \put(32,32){$\bfb^+$}
\put(25,30){\vector(1,0){0.8}} \put(35,30){\vector(-1,0){0.8}}
\put(10,30){\circle*{2}} \put(30,10){\circle*{2}} \put(32,7){$\bfb$}
\end{picture}
\end{figure}
  The images $l_1 = [\calH_1^0]$
and  $l_2 = [\calH_2^0]$  in $\boldA_2$ are the lines $(s_1 =
1)$ and  $(2s_1 + s_2 = 1)$ respectively.  These lines intersect in
the point $\bfb = [\bfv]$ given by $(s_1,s_2) = (1,-1).$

 Let us consider the evolution of 
$[\bfa_t]$ under time, as measured by $(s_1,s_2)$ coordinates
in $\boldA_2.$   It is easy to see
that $\eta_1\left((0,0)\right) < 0$ and  $\eta_2\left((0,0)\right) <
0$ as stated in Lemma \ref{PU}. 
As in Example \ref{prototype},  properties of the slope field
imply that any point in $(\bfs > 0)$ must  approach the
fixed point $\bfb^+,$ corresponding to the point $(1:0:1)$ in
$P^2(\boldR).$ 
\end{example}

\section{Soliton trajectories for the Ricci flow}\label{soltraj}

\subsection{Some solutions of nilsoliton trajectories}

Before stating the theorem, we note that all known soliton metrics
on nilpotent Lie algebra admit stably Ricci-diagonal bases. 
 It follows  from Theorem \ref{riccitensor} that
 any time an orthogonal  basis
  $\calB = \{ X_i\}$ has the property that the sets
$\{ J_{X_i}\}$ and $\{ \ad_{X_i}\}$ are orthogonal, the basis is 
stably Ricci-diagonal.  
In many other cases a stably Ricci-diagonal basis is guaranteed to exist:

\begin{prop}
Let $(\frakn, \sfQ)$ be a soliton metric nilpotent Lie algebra
with associated semisimple derivation $D = \Ric_{\sfQ} - \beta \Id.$  Suppose
that the eigenspaces for $D$ are all one-dimensional, and let
$\calB$ be a set of orthogonal eigenvectors.  The set $\calB$ is a 
stably Ricci-diagonal basis. 
\end{prop}

\begin{proof} 
 From \cite{heberinv} it is known that
the eigenvalues of $D$ are  positive and rational for all $i = 1, \ldots, n.$
 Write $\calB = \{\bfx_i\}_{i=1}^n$ where
$D(\bfx_i) = \lambda_i \bfx_i,$ with $\lambda_1 < \cdots < \lambda_n.$ 
 Because $D$ is a derivation, 
for all $i < j,$ the bracket 
$[\bfx_i, \bfx_j]$ is in the one-dimensional eigenspace for 
$\lambda_i + \lambda_j,$
so is a multiple of $\bfx_k$ for some $k > j.$
Hence the sets $\{J_{\bfx_i}\}$ and $\{\ad_{\bfx_i}\}$ are orthogonal,
and by Remark \ref{good}, the Ricci form is diagonal.  
\end{proof}

\begin{theorem}\label{soliton traj} Let $(\frakn,\sfQ)$ be a metric
nilpotent Lie algebra with  stably Ricci-diagonal basis  $\calB =
\{ \bfx_i \}_{i=1}^n$ with $|\Lambda_\calB| = m > 0.$  Suppose that
$\sfQ$ is a soliton inner
product with soliton constant $\beta.$
Let $\bfRic_{\calB} = (r_1, \ldots,
r_n)$  denote the Ricci vector for $(\frakn, \sfQ)$  relative to $\calB$.

 Let 
$\sfQ_t = \sum_{i=1}^n q_i \, dx^i \otimes dx^i$ be the solution to the 
Ricci flow with initial condition $\sfQ.$
Let $\bfa_t = (a_1(t), \ldots, a_m(t))$ be
the structure vector for $(\frakn,\sfQ_t)$  as in Equation \eqref{adef}.
Denote $a_1, \ldots, a_m$ also by $a_{jk}^l,$    where
$(j,k,l)$ is in $\Lambda_\calB.$
\begin{enumerate}% PART 1 - a_i
\item{
For $i=1, \ldots, m,$ the function $a_i$ is the solution 
\[ a_i(t) =  a_i(0)
\left( -2\beta t +  1\right)^{-1}\] to  the
differential  equation
$\D{
a_i^\prime(t)  = \frac{2\beta}{a_i(0)} a_i^2(t) }.$  The ray 
$\boldR\bfa_0$ is invariant under the flow,
with 
\[ \bfa_t = \frac{a_1(t)}{a_1(0)} \, \bfa_0.\]  
}\label{ai(t)-sol} 
\item{The Ricci form $\ric_{\sfQ_t}$ for $(\frakn,\sfQ_t)$ is diagonal
relative to 
$\calB$ 
with diagonal entries given by the Ricci
vector for $\sfQ_t:$
\[ \bfRic_{\calB} =  a_1(t) \left( 
-\frac{1}{2 a_1(0)}  \, \bfa_0^T Y \right)   \]
  }\label{Ric(t)-sol}
\item{The solution $\sfQ_t = \sum_{i=1}^n q_i \, dx^i \otimes dx_i$
 to the  Ricci flow for $(\frakn,\sfQ)$ is given by
\[
q_j(t) = 
q_j(0)  \left( -2\beta t + 1   \right)^{r_j/\beta} 
, 
\]
 for $j = 1, \ldots, n.$ }\label{q-soliton-solutions}
\item{Let  $E_{\text{min}}$ denote the  eigenspace for the minimal
eigenvalue of $\ric_{\sfQ}.$ The inner product $\sfQ$ collapses under
the Ricci flow to the equivalence class 
$\overline{\sfQ_\infty},$ where $\sfQ_\infty$ is a
 semidefinite symmetric bilinear form   that is positive definite on 
$E_{\text{min}}$ and is degenerate on any subspace properly
containing $E_{\text{min}}.$ 
}\label{soliton collapse}
\end{enumerate}
\end{theorem}
Before we prove the theorem, we illustrate it with an example.
\begin{example}\label{prototype-2}
 Let $(\frakn,\sfQ)$ be the five-dimensional
metric Lie algebra that with respect to an orthogonal basis $\calB = \{
\bfx_i\}_{i=1}^5$ has the following bracket relations:
\[ [\bfx_{1},\bfx_{3}]=\bfx_{4},  \qquad [\bfx_{1},\bfx_{4}]=\bfx_{5},
\qquad [\bfx_{2},\bfx_{3}]=\bfx_{5}   .\] 
By Theorem \ref{riccitensor}, $\calB$ is a Ricci diagonal basis. 
The set  $\Lambda_\calB$ is equal to 
\[\{(1,3,4), (1,4,5), (2,3,5)\},\] 
and the Gram matrix $U$ and the matrix
$PU$ are given by  
\[ U =\begin{bmatrix} 3 & 0 & 1 \\ 0 & 3 & 1 \\ 1 & 1 &
3 \end{bmatrix}  \qquad \text{and} \qquad  PU =\begin{bmatrix} 2 & -1
& -2 \\ -1 & 2 & -2  \end{bmatrix}. \] The kernel of $PU$ is spanned by
$\bfv = (2,2,1)^T;$ note that  $U\bfv = 7 (1,1,1)^T.$
 Therefore, by
Theorem \ref{Ua},  if the inner product $\sfQ$ which has $\calB$
as an orthogonal basis has 
structure vector $\bfa = (\frac{q_4}{q_1q_3},
\frac{q_5}{q_1q_4},\frac{q_5}{q_2q_3})$
equal to $(2,2,1)^T,$ then it  is  soliton with
soliton constant $\beta = -7/2.$  In that case the Ricci
vector is   $-\smallfrac{1}{2}(4,1,3,0,-3),$ and
the  derivation $D = \Ric - \beta \Id$ is represented by  $[D]_\calB,$
 the diagonal matrix $\diag(3,6,4,7,10).$ 

In this example, $a_1 = a_{13}^4, a_2 = a_{14}^5$ and $a_3 = a_{23}^5.$
  We let 
 $\sfQ = \sum_{i=1}^5 q_i dx^i \otimes dx^i$
with $q_1 = 1, q_2 = 4, q_3 = 1, q_4 = 2, q_5 = 4.$
The soliton 
metric nilpotent Lie algebra $(\frakn,\sfQ)$ has structure vector
 $\bfa_0  = (2,2,1)^T,$ 
 Ricci vector 
$-\frac{1}{2}(4,1,3,0,-3)$ and  soliton constant $\beta = -7/2.$ 

By Theorem \ref{soliton traj}, the flow for $\bfa_t$ is given by 
\[ a_1^\prime =  -\smallfrac{7}{2}a_1^2, \qquad  a_2^\prime =
-\smallfrac{7}{2}a_2^2, \qquad a_3^\prime =  -7 a_3^2, \]
 solutions of which are  
\[ a_1(t) =  2 \left(7 t + 1 \right)^{-1} \hskip
-8pt, a_2(t) = 2
\left(7 t +  1 \right)^{-1}
\hskip -8pt , a_3(t) = \left(7 t + 1 \right)^{-1} \hskip
-8pt .  \] A solution $\bfa_t$   takes values on the ray
$\boldR^+ ( 2, 2, 1);$ to be precise,  
\[ \bfa_t = \left(7 t + 1 \right)^{-1}  (2,2,1).\]

Solutions for the functions $q_1, \ldots, q_5$ are 
\begin{align*} q_1(t) &=   \left(  7t + 1
\right)^{4/7} \\
q_2(t) &=   4 \left(  7  t + 1\right)^{1/7}\\
 q_3(t)&=  
\left(  7 t + 1 \right)^{3/7} \\ 
 q_4(t)&=  2 \\
q_5(t)&=  4 \left(  7 t + 1 \right)^{-3/7}.
\end{align*}
Then
\[ q_1(t) \asymp t^{4/7}, q_2(t) \asymp t^{1/7},  q_3(t) \asymp t^{3/7},
q_4(t) \asymp t^{0}, \, \text{and} \, q_5(t) \asymp t^{-3/7},\] 
 We the notation 
$f(t) \asymp g(t)$  indicates that for functions $f(t)$ and
$g(t),$ the limit $\lim_{t \to \infty} \frac{f(t)}{g(t)}$ is a 
finite nonzero number.

The inner product collapses to a degenerate
volume-normalized  symmetric bilinear
form $\overline{\sfQ_\infty}$ supported on 
$\myspan\{\bfx_1\},$ which is the eigenspace for the  
 minimal eigenvalue of the semisimple derivation $D = \Ric - \beta \Id.$
\end{example}

Now we prove the theorem.

\begin{proof}[Proof of Theorem \ref{soliton traj}.] 
Since $(\frakn, \sfQ)$ is soliton, by Theorem \ref{Ua},   $U \bfa_0 =
-2\beta\onevector{m}$ for some $\beta < 0.$   
By Theorem \ref{ricci-system},
\[ a_i^\prime = - a_i (U\bfa_t)_i, \]
for $i=1, \ldots, m.$
 At the point $\bfa = \lambda
\bfa_0,$ for any $\lambda > 0$ and any $i = 1, \ldots, m,$ 
\begin{align}\label{aiprime}
 a_i^\prime(t) &= - a_i(t) \, (U\bfa)_i  \notag  \\
 &= - \lambda a_i(t) \,  (U\bfa_0)_i    \notag  \\
 &= - \lambda a_i(t)
\, (-2\beta  \onevector{m})_i  \notag  \\
&= 2 \beta \lambda  a_i(t).
\end{align} 
Therefore $\bfa^\prime(t) = 2\beta \lambda \bfa_t$ at all points 
$\lambda \bfa_0.$  Thus,
the ray $\boldR^+ \bfa_0$ is invariant under the flow.  Since
for $i = 1, \ldots, m,$ the function $a_i$ is given by 
 $a_i(t) = 
\smallfrac{a_i(0)}{a_1(0)} a_1(t)$ for all $t \ge 0,$ 
$\bfa_t = 
\smallfrac{a_1(t)}{a_1(0)} \bfa_0$ for all $t \ge 0.$

At the point  $\bfa_t = \lambda
\bfa_0,$  the value of  $\lambda$ is
$\smallfrac{a_1(t)}{a_1(0)},$ so Equation \eqref{aiprime} becomes
\[ a_i^\prime(t) = \frac{2  \beta }{a_1(0)} a_1(t) a_i(t) 
=   \frac{2\beta}{a_i(0)} a_i^2(t).\]
  Solutions  are
\[ a_i(t) = a_i(0)\left( -2\beta t +  1 \right)^{-1},\] for $i =
1, \ldots, m.$  This
proves Part \ref{ai(t)-sol}.  

Part \ref{Ric(t)-sol} is an immediate consequence of Equation
\eqref{Ricci vector} in Theorem \ref{riccitensor}.

Now we consider Part \ref{q-soliton-solutions} of the theorem.  In
order to compute $q_j(t),$ for $j = 1, \ldots, m,$ by Theorem 
\ref{ricci-system}, we
 need to solve the system of differential equations
$  (\ln q_j)^\prime = - 2 (\bfRic_\calB)_j,$ where
$j = 1, \ldots, n.$ 
Substituting the expressions for for $a_i(t)$ and $\bfRic_\calB$ 
from Parts \ref{ai(t)-sol} and
\ref{Ric(t)-sol}, we have 
\begin{align*} 
(\ln q_j)^\prime &= \sum_{\substack{(j,k,l) \in
\Lambda_\calB \\ (k,j,l) \in \Lambda_\calB }} a_{jk}^l - \sum_{(k,l,j)
\in \Lambda_\calB } a_{kl}^j \\  
&=  \sum_{\substack{(j,k,l) \in
\Lambda_\calB \\ (k,j,l) \in \Lambda_\calB }} ((a_{kl}^j(0)) (-2\beta + 1)^{-1} 
- \sum_{(k,l,j) \in \Lambda_\calB }  ((a_{jk}^l(0)) (-2\beta + 1)^{-1}. 
\end{align*} Integrating, we find that $\ln q_j$
is equal to 
\begin{equation*}  \sum_{\substack{(j,k,l) \in
\Lambda_\calB \\ (k,j,l) \in \Lambda_\calB }} \frac{a_{jk}^l}{-2\beta} \ln
\left(-2\beta t + 1 \right) - \sum_{(k,l,j) \in
\Lambda_\calB }  \frac{a_{kl}^j}{-2\beta} \ln(-2\beta t + 1)
\quad + C. 
\end{equation*} 
After exponentiating both sides and using that 
\[ r_j =  -\frac{1}{2}
\left( \sum_{\substack{(j,k,l) \in \Lambda_\calB \\ (k,j,l) \in
\Lambda_\calB }} a_{jk}^l -  \sum_{(k,l,j) \in \Lambda_\calB } 
a_{kl}^j \right)  \]
by
Theorem \ref{riccitensor},
 the desired expression for $q_j(t)$ is obtained. 

 Because inner product $\sfQ$ is soliton, the Lie algebra
$\frakn$ is orthogonally $\boldN$-graded by the eigenspaces of the 
derivation
associated to $\sfQ.$ 
Without loss of generality, let $r_1$ be the minimal eigenvalue of 
the Ricci form for $\sfQ,$ and let $E_{\text{min}}$ be the corresponding
eigenspace.  Then 
$ \lim_{t \to \infty} t^{-r_1 /\beta} \sfQ_t$
is positive definite on  $E_{\text{min}}$ but is degenerate on any subspace
properly containing  $E_{\text{min}}.$   This completes the proof of the 
theorem. 
\end{proof}

\section{Examples}\label{examples}

\subsection{Heisenberg metric Lie algebras}

The Heisenberg metric Lie algebras are among the most symmetric
nilpotent Lie algebras.   
\begin{example}  
 Let $\sfQ$ be an inner product on the $(2r+1)$-dimensional
Heisenberg algebra $\frakh_{2r+1}.$ From
Equation \eqref{ricxy}, it is easy to see that the center $\frakz$ is
the single positive eigenspace for the  Ricci endomorphism, and the
Ricci form is negative definite on $\frakz^\perp$.  A vector $\bfz$ spanning the center is an eigenvector for the Ricci
endomorphism.   By orthogonally block  diagonalizing the nondegenerate
skew-symmetric endomorphism $J_{\bfz}: \frakh_{2r+1} \to
\frakh_{2r+1}$ into one block of form $[0]$ and $r$  blocks of form
$( \begin{smallmatrix}  0 & c \\ -c & 0 \end{smallmatrix}),$ it is
possible to find an orthogonal Ricci eigenvector basis $\calB = \{
\bfx_i\}_{i=1}^{2r+1}$ of $\frakh_{2r+1}$ such that
  $\Lambda_\calB =
\{ (2i - 1,2i,2r+1) \, | \, i = 1, \ldots, r\}.$ The Gram matrix 
$U = (u_{ij})$
for $\frakh_{2r+1}$ relative to $\calB$ is the positive definite $r
\times r$ matrix defined by 
\[ u_{ij} = \begin{cases} 3 & i=j \\ 1 & i \ne j 
\end{cases}. 
\] The solution to  $U\bfv = \onevector{r}$ is a scalar multiple of  $\bfb =
\onevector{r}.$  Therefore, any inner product $\sfQ^\star$ with structure vector
that is a scalar multiple of 
$[1]$ is soliton.   The $(r-1) \times r$  matrix
$PU = (b_{ij})$ is of form \[ b_{ij} = \begin{cases} 2 & i=j \\ 0 & i
\ne j, 1 \le i, j \le m-1 \\ -2 & j = m 
\end{cases} 
\] and by Remark \ref{time change}, after a change of  variables to
$\bfs$, the trajectories for the projectivized Ricci flow  for
$\frakh_{2r+1}$ are encoded in the system of differential equations
\[(\ln s_i)^\prime = 2(1- s_i), \quad i=1, \ldots, r-1, \] which has
solution
\[ s_i(t) =  \frac{e^{2t}}{C_i + e^{2t}}, \quad i = 1, \ldots, r-1.\]
Clearly, $\bfs_t$ converges to $[1]$ as $t \to \infty$
 for all initial  conditions.  At
the limit point, all values of the structure constants are equal, so
that the limiting volume-normalized  metric Lie algebra
 is  the
Heisenberg Lie algebra endowed with  a volume-normalized soliton inner product.
\end{example}

\subsection{When the Lie algebra does not support a soliton inner product}

Next is an example of a seven-dimensional nilpotent metric Lie algebra
$(\frakn,\sfQ)$  such that the limit 
$\overline{(\frakn_\infty,\sfQ)}$ of  $\overline{(\frakn, \sfQ)}$ under the
projectivized Ricci flow 
$\psi_t: \calN_7 \to \calN_7$ has a limiting Lie algebra $\frakn_\infty$
that is  not isomorphic to the initial Lie algebra $\frakn.$
\begin{example}\label{R6}  Let $\frakn$ be a  Lie
algebra with basis $\calB= \{\bfx_i\}_{i=1}^{7}$ and with 
algebraic structure  determined by the bracket relations 
\begin{align*}   [\bfx_1,\bfx_i] & =  (\alpha_{1,i}^{i +1}) \bfx_{i +
1} \qquad \text{for $i = 2, \ldots , 6$ and} \\ [\bfx_2,\bfx_i] & =
(\alpha_{2,i}^{i +2}) \bfx_{i + 2} \qquad \text{for $i = 3, 4, 5$},
\end{align*}  where $\alpha_{1,i}^{i+1} \ne 0$ for  $i = 2, \ldots ,
6$ and $\alpha_{2,i}^{i+2} \ne 0$ for  $i = 2, \ldots , 5.$ No Lie
algebra of this form admits a soliton inner product (Theorem 34,
\cite{payne05}).

Take an initial inner product $\sfQ$ that is diagonal with respect
to $\calB.$  By  Theorem \ref{riccitensor},
$\calB$ is a  Ricci-diagonal basis, and by Remark 
\ref{good} $\calB$ remains Ricci-diagonal under rescalings of 
$\sfQ,$ so it is stably Ricci-diagonal.   Let $\bfa_t$ denote the
structure vector for the solution $\sfQ_t$ to the Ricci flow at time
$t.$  Recall that the entries
of $a_i$ are the squares of the nonzero structure constants $\alpha_{jk}^l.$
 Since no Lie algebra with $\bfa_t > 0$ can be soliton, the limit
\[ \bfa^\star = \lim_{t \to \infty} [\bfa_t] = (a_1 : a_2 : \cdots : a_8)\] 
in projective space
must have some entry $a_i$ equal to zero. 

The only elements $\bfx$  in $\frakn$ such that
the endomorphism $\ad_{\bfx}$
has rank three or more lie in $\myspan \{\bfx_1, \bfx_2\}.$
For this kind of element $\bfx = c_1 \bfx_1 + c_2 \bfx_2$,
where $c_1, c_2 \in \boldR,$ the adjoint map 
is represented with respect to $\calB$ by the matrix 
\[ [\ad_{c_1 \bfx_1 + c_2 \bfx_2}]_{\calB} =  
\begin{bmatrix} 
0 & 0 & 0 & 0 & 0 & 0 &  0 \\
0 & 0 & 0 & 0 & 0 & 0 &  0 \\
0 & c_1 \alpha_{1,2}^3 & 0 & 0 & 0 & 0 &  0 \\
0 &  0 & c_1\alpha_{1,3}^4 & 0 & 0 & 0 &  0 \\
0 & 0 &  c_2 \alpha_{2,3}^5  & c_1 \alpha_{1,4}^5 & 0 & 0 &  0 \\
0 & 0 & 0 &  c_2 \alpha_{2,4}^6 & c_1 \alpha_{1,5}^6 & 0 &  0 \\
0 & 0 & 0 & 0 &  c_2 \alpha_{2,5}^7  & c_1  \alpha_{1,6}^7 &  0 \\
\end{bmatrix}.
\]
In particular, there exist elements 
$\bfx_1$ and $\bfx_2$ such that the ranks of $\ad_{\bfx_1}$ and
$\ad_{\bfx_2}$ are five and three respectively.

We claim that if any structure constant becomes  zero 
as $[\bfa_t]$ approaches its limit, the
limiting Lie algebra $\frakn_\infty$ no longer has this property and
therefore is not isomorphic to the original Lie algebra $\frakn.$
The only way that $\frakn_\infty$ can have an element such that the rank
of $\ad_{c_1 \bfx_1 + c_2 \bfx_2}$ is five is if
 $\alpha_{1,i}^{i+1} \ne 0$ for $i=2, \ldots, 6.$  In order to have
an additional element $c_1 \bfx_1 + c_2 \bfx_2$ 
with the rank of  $\ad_{c_1 \bfx_1 + c_2 \bfx_2}$ 
equal to three, it is necessary  $\alpha_{2,i}^{i+2} \ne 0$ for 
$i=3,4,5.$  But then $\bfa^\star > 0,$ a contradiction.
  Therefore,  $\frakn_\infty$ can not be isomorphic to 
$\frakn.$
 
\end{example}

\bibliographystyle{amsalpha}
\bibliography{bibfile}

\end{document}